\documentclass[11pt]{amsart}
\usepackage{amsmath,amssymb}
\usepackage[margin=1.1in]{geometry}
\usepackage{tikz}
\usepackage[hidelinks]{hyperref}

\newtheorem{theorem}{Theorem}
\newtheorem{lemma}[theorem]{Lemma}
\newtheorem{proposition}[theorem]{Proposition}
\newtheorem{corollary}[theorem]{Corollary}
\newtheorem{conjecture}[theorem]{Conjecture}
\theoremstyle{remark}
\newtheorem{remark}[theorem]{Remark}

\newcommand{\np}{n^{+}}
\newcommand{\nz}{n^{0}}
\newcommand{\nm}{n^{-}}
\newcommand{\In}{\mathrm{In}}

\title{The signature of connected line graphs is unbounded}
\author{Luke Francis}
\email{lukefrancis09@icloud.com}
\author{Trevor Uptain}
\email{trevoruptain@gmail.com}
\date{\today}

\begin{document}

\begin{abstract}
Akbari, Elphick, Kumar, Pragada and Tang [Discrete Math.\ 349 (2026) 114953]
conjectured that for every connected graph $G$, the line graph of $G$ has at
most one more positive than negative adjacency eigenvalue; equivalently, the
signature of a connected line graph is at most $1$. We refute the conjecture
with two independently found counterexamples: a $14$-vertex cactus consisting
of two pentagons attached by bridges to adjacent vertices of a square, whose
line graph has inertia $(9,0,7)$ by an exact characteristic-polynomial
certificate, and a $48$-vertex triangle-free graph found by simulated
annealing and verified in exact rational arithmetic. Indeed, chaining copies
of the $14$-vertex graph yields connected graphs on $14k$ vertices whose line
graphs have signature $k+1$ for every $k \ge 1$. The signature of connected
line graphs is therefore unbounded, and no constant-bound repair of the
conjecture is possible.
\end{abstract}

\maketitle

\section{Introduction}

Let $G$ be a finite simple graph of order $n$ and size $m$, with adjacency
matrix $A(G)$. Write $\np(G)$, $\nz(G)$, $\nm(G)$ for the numbers of
positive, zero and negative eigenvalues of $A(G)$ counted with multiplicity,
$\In(G) = (\np,\nz,\nm)$ for the inertia, and
$s(G) = \np(G) - \nm(G)$ for the signature. In \cite{AEKPT}, Akbari et al.\
posed the following conjecture:

\begin{conjecture}[\cite{AEKPT}]\label{conj}
For any connected graph $G$, its line graph $L(G)$ satisfies
$\np(L(G)) \le \nm(L(G)) + 1$.
\end{conjecture}

The authors of \cite{AEKPT} verified the statement for all connected graphs
on at most $9$ vertices and for numerous graphs with up to $100$ edges,
proved it for several families including trees, and observed equality for the
odd cycles $C_{4k+1}$, the kayak paddle graphs $K(4,5,1)$ and $K(5,5,1)$, and
the $5$-prism. They noted that $s(L(C_5 \cup C_5)) = 2$, which is why
connectivity is required, and that the conjecture is open only in the sparse
range $n \le m \le 2n-2$.

The conjecture is consistent with, and stronger than, what was previously
known: Wang and Fan \cite{WF} proved that line graphs satisfy the
Ma--Yang--Li \cite{MYL} signature bound $s(L(G)) \le c_5(L(G))$, where $c_5$
counts cycles of length $\equiv 1 \pmod 4$. That bound permits line-graph signatures
larger than $1$; the content of Conjecture~\ref{conj} was that they
nevertheless do not occur. They do.

In Section~\ref{sec:14} we present a $14$-vertex counterexample with an exact
hand-checkable certificate; Section~\ref{sec:48} gives a $48$-vertex
counterexample found by simulated annealing; and Section~\ref{sec:unbounded}
proves a bridge lemma from which the unboundedness of the signature of
connected line graphs follows.

\section{A 14-vertex counterexample with an exact certificate}\label{sec:14}

Let $G_{14}$ have vertex set
$\{a,b,c,d\} \cup \{u_0,\dots,u_4\} \cup \{v_0,\dots,v_4\}$ and edge set
consisting of the square $abcda$, the pentagons $u_0u_1u_2u_3u_4u_0$ and
$v_0v_1v_2v_3v_4v_0$, and the bridges $au_0$ and $bv_0$: two pentagons
attached by bridges at adjacent vertices of a square
(Figure~\ref{fig:g14}). Then $G_{14}$ is connected with $14$ vertices and
$16$ edges; its graph6 string is \texttt{Ml\_GGCHO??\_@?@?C\_} (the second
character is a lowercase letter l, not a digit), and machine-readable copies
of a numeric edge list and the graph6 string are provided in the ancillary
files.

\begin{figure}[ht]
\centering
\begin{tikzpicture}[scale=1.05,
  v/.style={circle,fill,inner sep=1.6pt},
  lbl/.style={font=\small}]
 
  \coordinate (a) at (-0.9,0.6);
  \coordinate (b) at (0.9,0.6);
  \coordinate (c) at (0.9,-1.2);
  \coordinate (d) at (-0.9,-1.2);
 
  \coordinate (u0) at (-2.2,0.6);
  \coordinate (u1) at (-2.891,1.551);
  \coordinate (u2) at (-4.009,1.188);
  \coordinate (u3) at (-4.009,0.012);
  \coordinate (u4) at (-2.891,-0.351);
 
  \coordinate (v0) at (2.2,0.6);
  \coordinate (v1) at (2.891,1.551);
  \coordinate (v2) at (4.009,1.188);
  \coordinate (v3) at (4.009,0.012);
  \coordinate (v4) at (2.891,-0.351);
 
  \draw (a) -- (b) -- (c) -- (d) -- (a);
  \draw (u0) -- (u1) -- (u2) -- (u3) -- (u4) -- (u0);
  \draw (v0) -- (v1) -- (v2) -- (v3) -- (v4) -- (v0);
  \draw (a) -- (u0);
  \draw (b) -- (v0);
 
  \foreach \p in {a,b,c,d,u0,u1,u2,u3,u4,v0,v1,v2,v3,v4}
    \node[v] at (\p) {};
 
  \node[lbl,above left] at (a) {$a$};
  \node[lbl,above right] at (b) {$b$};
  \node[lbl,below right] at (c) {$c$};
  \node[lbl,below left] at (d) {$d$};
  \node[lbl,below] at (u0) {$u_0$};
  \node[lbl,above] at (u1) {$u_1$};
  \node[lbl,left] at (u2) {$u_2$};
  \node[lbl,left] at (u3) {$u_3$};
  \node[lbl,below] at (u4) {$u_4$};
  \node[lbl,below] at (v0) {$v_0$};
  \node[lbl,above] at (v1) {$v_1$};
  \node[lbl,right] at (v2) {$v_2$};
  \node[lbl,right] at (v3) {$v_3$};
  \node[lbl,below] at (v4) {$v_4$};
\end{tikzpicture}
\caption{The graph $G_{14}$: two pentagons attached by bridges at adjacent
vertices of a square.}
\label{fig:g14}
\end{figure}
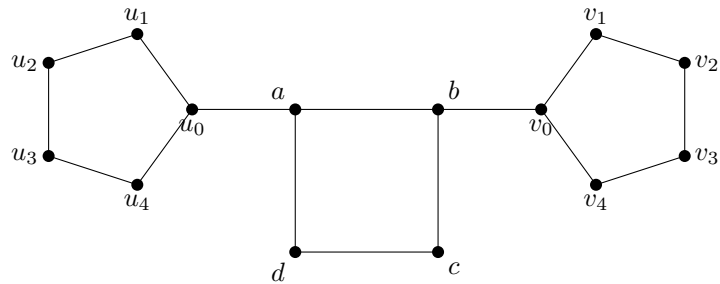

\begin{proposition}\label{prop:14}
$\In(L(G_{14})) = (9,0,7)$. Hence $s(L(G_{14})) = 2$ and
Conjecture~\ref{conj} is false.
\end{proposition}

\begin{proof}
Let $H = L(G_{14})$, so $A(H)$ is a $16 \times 16$ symmetric $0$--$1$ matrix.
A symmetry decomposition of $A(H)$ with respect to the two pentagon
reflections and the reflection exchanging the pentagons splits
$\mathbb{R}^{16}$ into invariant subspaces of dimensions $2, 2, 7, 5$, and
taking determinants of the restriction matrices yields the exact
factorization
\begin{equation}\label{eq:charpoly}
\chi_H(x) = (x-2)(x-1)(x+2)^2 (x^2+x-1)^2 (x^3-2x^2-4x+1)(x^5-x^4-6x^3+3x^2+7x-2);
\end{equation}
the explicit bases and restriction matrices are given in the ancillary
files. All roots are real. The linear factors contribute two positive and
two negative roots, and $(x^2+x-1)^2$ contributes two of each sign. For
$f(x) = x^3-2x^2-4x+1$, Descartes' rule applied to $f(-x)$ gives exactly one
negative root, so $f$ has two positive roots. For
$g(x) = x^5-x^4-6x^3+3x^2+7x-2$, Descartes' rule gives $\np \in \{3,1\}$ and
$\nm \in \{2,0\}$ independently, leaving the four possibilities $(3,2)$,
$(3,0)$, $(1,2)$, $(1,0)$; since all five roots are real and nonzero, the
counts must sum to $5$, which leaves only $(3,2)$. Totaling: $\np = 9$,
$\nm = 7$, and no factor vanishes at zero, so $\nz = 0$.
\end{proof}

The factorization \eqref{eq:charpoly} and the inertia were additionally
verified by exact symbolic computation and by exact rational congruence
diagonalization (Sylvester's law of inertia); see the ancillary files.

\section{A 48-vertex counterexample from simulated annealing}\label{sec:48}

Let $G_{48}$ be the graph on vertex set $\{0,\dots,47\}$ with the $50$ edges
\begin{center}\small\tt
(0,1) (0,2) (0,3) (1,5) (1,12) (1,13) (2,12) (3,16) (3,24) (4,5)\\
(4,6) (5,7) (5,8) (6,31) (7,25) (9,10) (9,11) (11,14) (11,15) (13,22)\\
(13,28) (14,26) (15,27) (16,17) (17,29) (17,30) (18,19) (19,20) (20,31) (21,22)\\
(21,23) (24,43) (25,26) (26,27) (28,32) (29,37) (30,38) (30,39) (32,46) (33,34)\\
(33,35) (34,36) (36,43) (37,39) (38,40) (38,41) (40,42) (41,44) (44,45) (46,47)
\end{center}
$G_{48}$ is connected and triangle-free, with maximum degree $4$, cycle rank
$3$, and minimum cycle basis of lengths $(4,5,5)$. Its line graph $L(G_{48})$
has $50$ vertices and $\In(L(G_{48})) = (26,0,24)$, hence $s(L(G_{48})) = 2$.
The $2$-core of $G_{48}$, however, has
$\In(L(\operatorname{core}_2(G_{48}))) = (11,0,10)$ and line-graph signature
$1$. Exact pruning shows that the additional unit of signature is supplied by
the single pendant edge $(5,8)$; the remaining hanging paths are
signature-neutral pendant-$P_2$ extensions as in Lemma~\ref{lem:pendant}.
Moreover, after suppressing degree-$2$ paths, the subdivision-length profile
of the $2$-core is $(1,3,3,4,5,5)$, whereas the profile of $G_{14}$ is
$(1,1,1,3,5,5)$. Thus the two graphs share the cycle-basis length multiset
$(4,5,5)$ but not the same reduced skeleton or the same mechanism. This was
verified three independent ways: floating-point eigendecomposition (the
spectrum has no eigenvalue within $0.024$ of zero), exact rational congruence
diagonalization, and the signless Laplacian identity of
Lemma~\ref{lem:spectrum} below. The graph was found by simulated annealing
seeded at the equality cases of \cite{AEKPT}; a second, independent
size-minimizing run converged to a different $48$-vertex graph with the same
invariants.

\section{The bridge lemma and unbounded signature}\label{sec:unbounded}

Throughout, $Q(G) = A(G) + D(G)$ denotes the signless Laplacian and, for a
graph $G$ with $2 \notin \operatorname{spec} Q(G)$, we write $\varphi_G(x)$
for the $(x,x)$ entry of $(Q(G) - 2I)^{-1}$.

\begin{lemma}\label{lem:spectrum}
For any graph $G$ (not necessarily connected) with $m \ge n$, the spectrum of
$A(L(G))$ consists of $q_i - 2$ over all $n$ signless Laplacian eigenvalues
$q_i$ of $G$, together with $m-n$ additional copies of $-2$. Consequently
$s(L(G)) = 2\,\#\{q_i > 2\} + \#\{q_i = 2\} - m$.
\end{lemma}

\begin{proof}
With $N$ the unsigned vertex--edge incidence matrix, $N^{T}N = A(L(G)) + 2I$
and $NN^{T} = Q(G)$ have the same nonzero spectra; zero eigenvalues of either
product correspond to the eigenvalue $-2$ of $A(L(G))$ after the shift, and
the multiplicity count gives the statement.
\end{proof}

\begin{lemma}[Bridge lemma]\label{lem:bridge}
Let $G_1, G_2$ be disjoint graphs with $m_i \ge n_i$ such that $2$ is not a
signless Laplacian eigenvalue of either, and let $u \in G_1$, $v \in G_2$.
If $\varphi_{G_1}(u) + \varphi_{G_2}(v) > -1$, then the graph
$G = G_1 \cup G_2 + uv$ satisfies
\[
s(L(G)) = s(L(G_1)) + s(L(G_2)) - 1,
\]
and $2$ is again not a signless Laplacian eigenvalue of $G$.
\end{lemma}

\begin{proof}
Adding the edge $uv$ adds $ww^{T}$ to $Q$, where $w = e_u + e_v$ and
$Q = Q(G_1) \oplus Q(G_2)$. Since
$\det(Q + t\,ww^{T} - 2I) = \det(Q-2I)\bigl(1 + t\,w^{T}(Q-2I)^{-1}w\bigr)$
is linear in $t$, and rank-one positive semidefinite updates move eigenvalues
weakly upward, no eigenvalue crosses or lands on $2$ for $t \in (0,1]$
provided $w^{T}(Q-2I)^{-1}w > -1$. As $(Q-2I)^{-1}$ is block diagonal, the
cross terms vanish and $w^{T}(Q-2I)^{-1}w = \varphi_{G_1}(u) +
\varphi_{G_2}(v)$. Thus $\#\{q>2\}$ and $\#\{q=2\} = 0$ are unchanged while
$m$ increases by $1$; Lemma~\ref{lem:spectrum} gives the claim.
\end{proof}

\begin{theorem}\label{thm:main}
Let $B$ be a graph with $m \ge n$, $s(L(B)) = \sigma$,
$2 \notin \operatorname{spec} Q(B)$, and a vertex $v$ with
$\beta := \varphi_B(v) > 0$. For $k \ge 1$ let $B_k$ consist of $k$ disjoint
copies of $B$ with a bridge from vertex $v$ of copy $i$ to vertex $v$ of copy
$i+1$. Then
\[
s(L(B_k)) = k\sigma - (k-1) \qquad \text{for all } k \ge 1.
\]
\end{theorem}

\begin{proof}
Let $a_k = \varphi_{B_k}(v_k)$, where $v_k$ is vertex $v$ of the last copy.
By the Sherman--Morrison formula \cite{SM} applied to
$(Q(B_k)-2I) \oplus (Q(B)-2I) + ww^{T}$, the resolvent entry at the new
attachment vertex satisfies
\[
a_1 = \beta, \qquad a_{k+1} = \beta - \frac{\beta^2}{1 + a_k + \beta}.
\]
If $a_k > 0$ then $\beta^2/(1+a_k+\beta) < \beta^2/(1+\beta)$, so
$a_{k+1} > \beta - \beta^2/(1+\beta) = \beta/(1+\beta) > 0$; by induction
$a_k > 0$ for all $k$, hence the crossing condition $1 + a_k + \beta > 1 > 0$
of Lemma~\ref{lem:bridge} holds at every step (applied to $G_1 = B_k$,
$G_2 = B$), and $2$ never enters the signless Laplacian spectrum. Each bridge
therefore costs exactly one unit of signature while each copy contributes
$\sigma$.
\end{proof}

\begin{corollary}\label{cor:14k}
For every $k \ge 1$ there is a connected graph on $14k$ vertices whose line
graph has signature $k + 1$. In particular
$\sup\{\, s(L(G)) : G \text{ connected} \,\} = \infty$, and no constant-bound
repair of Conjecture~\ref{conj} is possible.
\end{corollary}

\begin{proof}
Apply Theorem~\ref{thm:main} to $G_{14}$ with $v$ the pentagon vertex
labeled $v_4$ (numeric label 13 in the ancillary edge list): exact computation gives $\det(Q(G_{14}) - 2I) = -4$, so
$2 \notin \operatorname{spec} Q(G_{14})$, and
$\beta = \varphi_{G_{14}}(v_4) = \tfrac12 > 0$. By
Proposition~\ref{prop:14}, $s(L(G_{14})) = 2$, and so $s(L(B_k)) = k+1$ by
Theorem~\ref{thm:main}.
\end{proof}

The same construction applies to $B = G_{48}$, where
$\det(Q(G_{48})-2I) = 4$ and the $48$ resolvent entries take the values
$\tfrac12, \tfrac32, \tfrac52, \tfrac72, \tfrac92$ with multiplicities
$19, 15, 8, 4, 2$; chains of $G_{48}$ were the route by which
Theorem~\ref{thm:main} was first found and verified.

A complementary tool grows counterexamples without changing the signature:

\begin{lemma}\label{lem:pendant}
Let $F$ be any graph, $v \in V(F)$, and form $F'$ by adding new vertices
$x,y$ and edges $vx, xy$. Then $\In(L(F')) = \In(L(F)) + (1,0,1)$, so
$s(L(F')) = s(L(F))$.
\end{lemma}

\begin{proof}
Order the vertices of $L(F')$ by the old edges, then $vx$, then $xy$. With
$B = A(L(F))$ and $b$ the indicator of old edges at $v$,
\[
A(L(F')) = \begin{pmatrix} B & C \\ C^{T} & J \end{pmatrix},
\quad C = \begin{pmatrix} b & 0 \end{pmatrix},
\quad J = \begin{pmatrix} 0 & 1 \\ 1 & 0 \end{pmatrix}.
\]
Since $J^{-1} = J$ and $CJ^{-1}C^{T} = 0$, Schur-complement congruence gives
$A(L(F')) \cong_{\mathrm{cong}} B \oplus J$, and $J$ contributes one positive
and one negative eigenvalue.
\end{proof}

\section{Remarks}

\begin{remark}[Minimality]
An exhaustive enumeration of connected subcubic graphs found no
counterexample through $13$ vertices, and at $14$ vertices exactly one
isomorphism class with line-graph signature $2$, represented by $G_{14}$.
This is a screening result among subcubic graphs, not a claim that $14$ is
the minimum order over all connected graphs. Determining the minimum order of
a connected graph with $s(L(G)) = t$ for each $t \ge 2$ remains open.
\end{remark}

\begin{remark}[Mechanism]
The two base graphs share the minimum-cycle-basis length multiset $(4,5,5)$,
but their reduced skeletons and signature mechanisms differ. The graph
$G_{14}$ is its own $2$-core and already has line-graph signature $2$. By
contrast, the $2$-core of $G_{48}$ has line-graph signature $1$; a single
pendant edge raises the signature to $2$, while the remaining hanging paths
preserve signature by Lemma~\ref{lem:pendant}. Their suppressed subdivision
profiles are respectively $(1,1,1,3,5,5)$ and $(1,3,3,4,5,5)$. In
particular, cycle-basis lengths alone do not determine line-graph signature.
A structural explanation encompassing both mechanisms would be interesting.
\end{remark}

\begin{remark}[Prior work and independent discovery]
After submitting the first version, we learned of independent contemporaneous
work by Andrea Paone. Paone's Version~1 deposit, dated July~22, 2026
\cite{PaoneV1}, gives a graph isomorphic to $G_{14}$ with the same exact
line-graph inertia. It is the earliest public record of this example known to
us. Paone's Version~2, dated July~24, 2026 \cite{PaoneV2}, independently
proves that line-graph signature is unbounded even for connected simple planar
subcubic cactus graphs, using a rooted-module attachment construction.

Our $14$-vertex example was communicated privately to the authors of
\cite{AEKPT} on July~20, and our bridge-based unbounded family was described
in correspondence dated July~21--23; our arXiv submission was made on
July~24. We therefore describe the results as independent and contemporaneous
and make no exclusive priority claim. The note of Chen and Li \cite{CL}
refutes the different quadratic inertia inequality
$2\np(G) \le \nm(G)(\nm(G)+1)$ and a related order inequality from the same
paper; it does not address the line-graph conjecture.
\end{remark}

\begin{remark}[Provenance]
In their concluding remarks, the authors of \cite{AEKPT} noted that these
conjectures are well suited to AI-based counterexample search. Both
counterexamples arose from precisely such searches, conducted independently
and with different systems: the $14$-vertex example via a search assisted by
ChatGPT~5.6~Pro (OpenAI), the $48$-vertex example and the results of
Section~\ref{sec:unbounded} via a search and proof development assisted by
Claude Fable 5 (Anthropic), with all computations reproduced by separately
coded audits in exact arithmetic. That independent AI-assisted searches converged in the same week on
counterexamples with the same minimum-cycle-basis lengths $(4,5,5)$, while
the graphs have different reduced skeletons and achieve signature $2$ by
different mechanisms, may itself be of interest. Paone's independent searches
and exact certificates are described in \cite{PaoneV1,PaoneV2}.
\end{remark}

\section*{Acknowledgments}
The two counterexamples were found independently in July 2026 and
communicated to the authors of \cite{AEKPT}, who kindly introduced us.
Proposition~\ref{prop:14}, Lemma~\ref{lem:pendant} and the $14$-vertex
example are due to the first author; the $48$-vertex example and
Lemmas~\ref{lem:spectrum} and~\ref{lem:bridge}, Theorem~\ref{thm:main} and
Corollary~\ref{cor:14k} are due to the second author. We are grateful
to Hitesh Kumar for rapid, generous correspondence, for introducing the
authors, and for detailed comments on an earlier draft, and to S.~Akbari,
C.~Elphick, S.~Pragada and Q.~Tang for their paper and the invitation it
contains.

\end{document}